\documentclass{article}
\usepackage{amsmath}
\usepackage{amsfonts}

\newtheorem{theorem}{Theorem}

\newtheorem{lemma}[theorem]{Lemma}

\newtheorem{proposition}[theorem]{Proposition}
\newtheorem{remark}[theorem]{Remark}

\newenvironment{proof}[1][Proof]{\noindent\textbf{#1.} }{\ \rule{0.5em}{0.5em}}
\input{tcilatex}

\begin{document}

\title{Testing in the Presence of Nuisance Parameters: Some Comments on
Tests Post-Model-Selection and Random Critical Values}
\author{Hannes Leeb and Benedikt M. P\"{o}tscher \\
Department of Statistics, University of Vienna}
\date{ Preliminary version: May 25, 2012\\
First version: September 20, 2012\\
Second version: May 18, 2013\\
This version: May 15, 2014}
\maketitle

\begin{abstract}
We point out that the ideas underlying some test procedures recently
proposed for testing post-model-selection (and for some other test problems)
in the econometrics literature have been around for quite some time in the
statistics literature. We also sharpen some of these results in the
statistics literature. Furthermore, we show that some intuitively appealing
testing procedures, that have found their way into the econometrics
literature, lead to tests that do not have desirable size properties, not
even asymptotically.
\end{abstract}

\section{Introduction}

Suppose we have a sequence of statistical experiments given by a family of
probability measures $\left\{ P_{n,\alpha ,\beta }:\alpha \in A,\beta \in
B\right\} $ where $\alpha $ is a "parameter of interest", and $\beta $ is a
"nuisance-parameter". Often, but not always, $A$ and $B$ will be subsets of
Euclidean space. Suppose the researcher wants to test the null-hypothesis $%
H_{0}:\alpha =\alpha _{0}$ using the real-valued test-statistic $%
T_{n}(\alpha _{0})$, with large values of $T_{n}(\alpha _{0})$ being taken
as indicative for violation of $H_{0}$.\footnote{%
This framework obviously allows for "one-sided" as well as for "two-sided"
alternatives (when these concepts make sense) by a proper definition of the
test statistic.} Suppose further that the distribution of $T_{n}(\alpha
_{0}) $ under $H_{0}$ depends on the nuisance parameter $\beta $. This leads
to the key question: How should the critical value then be chosen? [Of
course, if another, pivotal, test-statistic is available, this one could be
used. However, we consider here the case where a (non-trivial) pivotal
test-statistic either does not exist or where the researcher -- for better
or worse -- insists on using $T_{n}(\alpha _{0})$.] In this situation a
standard way (see, e.g., Bickel and Doksum (1977), p.170) to deal with this
problem is to choose as critical value 
\begin{equation}
c_{n,\sup }(\delta )=\sup_{\beta \in B}c_{n,\beta }(\delta ),  \label{sup}
\end{equation}%
where $0<\delta <1$ and where $c_{n,\beta }(\delta )$ satisfies $P_{n,\alpha
_{0},\beta }\left( T_{n}(\alpha _{0})>c_{n,\beta }(\delta )\right) =\delta $
for each $\beta \in B$, i.e., $c_{n,\beta }(\delta )$ is \emph{a }$(1-\delta
)$-quantile of the distribution of $T_{n}(\alpha _{0})$ under $P_{n,\alpha
_{0},\beta }$. [We assume here the existence of such a $c_{n,\beta }(\delta
) $, but we do not insist that it is chosen as the smallest possible number
satisfying the above condition, although this will usually be the case.] In
other words, $c_{n,\sup }(\delta )$ is the "worst-case" critical value.
While the resulting test, which rejects $H_{0}$ for 
\begin{equation}
T_{n}(\alpha _{0})>c_{n,\sup }(\delta ),  \label{cons}
\end{equation}%
certainly is a level $\delta $ test (i.e., has size $\leq \delta $), the
conservatism caused by taking the supremum in (\ref{sup}) will often result
in poor power properties, especially for values of $\beta $ for which $%
c_{n,\beta }(\delta )$ is much smaller than $c_{n,\sup }(\delta )$. The test
obtained from (\ref{sup}) and (\ref{cons}) above (more precisely, an
asymptotic variant thereof) is what Andrews and Guggenberger (2009) call a
"size-corrected fixed critical value" test.\footnote{%
While Andrews and Guggenberger (2009) do not consider a finite-sample
framework but rather a "moving-parameter" asymptotic framework, the
underlying idea is nevertheless exactly the same.}

An alternative idea, which has some intuitive appeal and which is much less
conservative, is to use $c_{n,\hat{\beta}_{n}}(\delta )$ as a random
critical value, where $\hat{\beta}_{n}$ is an estimator for $\beta $ (taking
its values in $B$), and to reject $H_{0}$ if 
\begin{equation}
T_{n}(\alpha _{0})>c_{n,\hat{\beta}_{n}}(\delta )  \label{par_boot}
\end{equation}%
obtains (measurability of $c_{n,\hat{\beta}_{n}}(\delta )$ being assumed).
This choice of critical value can be viewed as a parametric bootstrap
procedure. Versions of $c_{n,\hat{\beta}_{n}}(\delta )$ have been considered
by Williams (1970) or, more recently, by Liu (2011). However, 
\begin{equation*}
P_{n,\alpha _{0},\beta }\left( T_{n}(\alpha _{0})>c_{n,\hat{\beta}%
_{n}}(\delta )\right) \geq P_{n,\alpha _{0},\beta }\left( T_{n}(\alpha
_{0})>c_{n,\sup }(\delta )\right)
\end{equation*}%
clearly holds for every $\beta $, indicating that the test using the random
critical value $c_{n,\hat{\beta}_{n}}(\delta )$ may \emph{not} be a level $%
\delta $ test, but may have size larger than $\delta $. This was already
noted by Loh (1985). A precise result in this direction, which is a
variation of Theorem 2.1 in Loh (1985), is as follows.

\begin{proposition}
\label{prop_1}Suppose that there exists a $\beta _{n}^{\max }=\beta
_{n}^{\max }(\delta )$ such that $c_{n,\beta _{n}^{\max }}(\delta
)=c_{n,\sup }(\delta )$. Then%
\begin{equation}
P_{n,\alpha _{0},\beta _{n}^{\max }}\left( c_{n,\hat{\beta}_{n}}(\delta
)<T_{n}(\alpha _{0})\leq c_{n,\sup }(\delta )\right) >0  \label{pos}
\end{equation}%
implies 
\begin{equation}
\sup_{\beta \in B}P_{n,\alpha _{0},\beta }\left( T_{n}(\alpha _{0})>c_{n,%
\hat{\beta}_{n}}(\delta )\right) >\delta ,  \label{overshoot}
\end{equation}%
i.e., the test using the random critical value $c_{n,\hat{\beta}_{n}}(\delta
)$ does not have level $\delta $. More generally, if $\hat{c}_{n}$ is any
random critical value satisfying $\hat{c}_{n}\leq c_{n,\beta _{n}^{\max
}}(\delta )(=c_{n,\sup }(\delta ))$ with $P_{n,\alpha _{0},\beta _{n}^{\max
}}$-probability $1$, then (\ref{pos}) still implies (\ref{overshoot}) if in
both expressions $c_{n,\hat{\beta}_{n}}(\delta )$ is replaced by $\hat{c}%
_{n} $. [The result continues to hold if the random critical value $\hat{c}%
_{n}$ also depends on some additional randomization mechanism.]
\end{proposition}

\begin{proof}
Observe that $c_{n,\hat{\beta}_{n}}(\delta )\leq c_{n,\sup }(\delta )$
always holds. But then the l.h.s. of (\ref{overshoot}) is bounded from below
by 
\begin{eqnarray*}
&&P_{n,\alpha _{0},\beta _{n}^{\max }}\left( T_{n}(\alpha _{0})>c_{n,\hat{%
\beta}_{n}}(\delta )\right) \\
&=&P_{n,\alpha _{0},\beta _{n}^{\max }}\left( T_{n}(\alpha _{0})>c_{n,\sup
}(\delta )\right) +P_{n,\alpha _{0},\beta _{n}^{\max }}\left( c_{n,\hat{\beta%
}_{n}}(\delta )<T_{n}(\alpha _{0})\leq c_{n,\sup }(\delta )\right) \\
&=&P_{n,\alpha _{0},\beta _{n}^{\max }}\left( T_{n}(\alpha _{0})>c_{n,\beta
_{n}^{\max }}(\delta )\right) +P_{n,\alpha _{0},\beta _{n}^{\max }}\left(
c_{n,\hat{\beta}_{n}}(\delta )<T_{n}(\alpha _{0})\leq c_{n,\sup }(\delta
)\right) \\
&=&\delta +P_{n,\alpha _{0},\beta _{n}^{\max }}\left( c_{n,\hat{\beta}%
_{n}}(\delta )<T_{n}(\alpha _{0})\leq c_{n,\sup }(\delta )\right) >\delta ,
\end{eqnarray*}%
the last inequality holding in view of (\ref{pos}). The proof for the second
claim is completely analogous.
\end{proof}

\bigskip

To better appreciate condition (\ref{pos}) consider the case where $%
c_{n,\beta }(\delta )$ is uniquely maximized at $\beta _{n}^{\max }$ and $%
P_{n,\alpha _{0},\beta _{n}^{\max }}(\hat{\beta}_{n}\neq \beta _{n}^{\max })$
is positive. Then 
\begin{equation*}
P_{n,\alpha _{0},\beta _{n}^{\max }}(c_{n,\hat{\beta}_{n}}(\delta
)<c_{n,\sup }(\delta ))>0
\end{equation*}%
holds and therefore we can expect condition (\ref{pos}) to be satisfied,
unless there exists a quite strange dependence structure between $\hat{\beta}%
_{n}$ and $T_{n}(\alpha _{0})$. The same argument applies in the more
general situation where there are multiple maximizers $\beta _{n}^{\max }$
of $c_{n,\beta }(\delta )$ as soon as $P_{n,\alpha _{0},\beta _{n}^{\max }}(%
\hat{\beta}_{n}\notin \arg \max c_{n,\beta }(\delta ))>0$ holds for one of
the maximizers $\beta _{n}^{\max }$.

In the same vein, it is also useful to note that Condition (\ref{pos}) can
equivalently be stated as follows: The conditional cumulative distribution
function $P_{n,\alpha _{0},\beta _{n}^{\max }}(T_{n}(\alpha _{0})\leq \cdot
\mid \hat{\beta}_{n})$ of $T_{n}(\alpha _{0})$ given $\hat{\beta}_{n}$ puts
positive mass on the interval $(c_{n,\hat{\beta}_{n}}(\delta ),c_{n,\sup
}(\delta )]$ for a set of $\hat{\beta}_{n}$'s that has positive probability
under $P_{n,\alpha _{0},\beta _{n}^{\max }}$. [Also note that Condition (\ref%
{pos}) implies that $c_{n,\hat{\beta}_{n}}(\delta )<c_{n,\sup }(\delta )$
must hold with positive $P_{n,\alpha _{0},\beta _{n}^{\max }}$-probability.]
A sufficient condition for this then clearly is that for a set of $\hat{\beta%
}_{n}$'s of positive $P_{n,\alpha _{0},\beta _{n}^{\max }}$-probability we
have that (i) $c_{n,\hat{\beta}_{n}}(\delta )<c_{n,\sup }(\delta )$, and
(ii) the conditional cumulative distribution function $P_{n,\alpha
_{0},\beta _{n}^{\max }}(T_{n}(\alpha _{0})\leq \cdot \mid \hat{\beta}_{n})$
puts positive mass on \emph{every} non-empty interval. The analogous result
holds for the case where $\hat{c}_{n}$ replaces $c_{n,\hat{\beta}%
_{n}}(\delta )$ (and conditioning is w.r.t. $\hat{c}_{n}$), see Lemma \ref%
{help} in the Appendix for a formal statement.

The observation, that the test (\ref{par_boot}) based on the random critical
value $c_{n,\hat{\beta}_{n}}(\delta )$ typically will not be a level $\delta 
$ test, has led Loh (1985) and subsequently Berger and Boos (1994) and
Silvapulle (1996) to consider the following procedure (or variants thereof)
which leads to a level $\delta $ test that is somewhat less "conservative"
than the test given by (\ref{cons}):\thinspace \footnote{\label{foot}Loh
(1985) actually considers the random critical value $c_{n,\eta
_{n},Loh^{\ast }}(\delta )$ given by $\sup_{\beta \in I_{n}}c_{n,\beta
}(\delta )$, which typically does not lead to a level $\delta $ test in
finite samples in view of Proposition \ref{prop_1} (since $c_{n,\eta
_{n},Loh^{\ast }}(\delta )\leq c_{n,\sup }(\delta )$). However, Loh (1985)
focuses on the case where $\eta _{n}\rightarrow 0$ and shows that then the
size of the test converges to $\delta $; that is, the test is asymptotically
level $\delta $ if $\eta _{n}\rightarrow 0$. See also Remark \ref{rem}.} Let 
$I_{n}$ be a random set in $B$ satisfying%
\begin{equation*}
\inf_{\beta \in B}P_{n,\alpha _{0},\beta }\left( \beta \in I_{n}\right) \geq
1-\eta _{n},
\end{equation*}%
where $0\leq \eta _{n}<\delta $. I.e., $I_{n}$ is a confidence set for the
nuisance parameter $\beta $ with infimal coverage probability not less than $%
1-\eta _{n}$ (provided $\alpha =\alpha _{0}$). Define a random critical
value via%
\begin{equation}
c_{n,\eta _{n},Loh}(\delta )=\sup_{\beta \in I_{n}}c_{n,\beta }(\delta -\eta
_{n}).  \label{Loh}
\end{equation}%
Then we have 
\begin{equation*}
\sup_{\beta \in B}P_{n,\alpha _{0},\beta }\left( T_{n}(\alpha
_{0})>c_{n,\eta _{n},Loh}(\delta )\right) \leq \delta .
\end{equation*}%
This can be seen as follows: For every $\beta \in B$%
\begin{eqnarray*}
P_{n,\alpha _{0},\beta }\left( T_{n}(\alpha _{0})>c_{n,\eta _{n},Loh}(\delta
)\right) &=&P_{n,\alpha _{0},\beta }\left( T_{n}(\alpha _{0})>c_{n,\eta
_{n},Loh}(\delta ),\beta \in I_{n}\right) \\
&&+P_{n,\alpha _{0},\beta }\left( T_{n}(\alpha _{0})>c_{n,\eta
_{n},Loh}(\delta ),\beta \notin I_{n}\right) \\
&\leq &P_{n,\alpha _{0},\beta }\left( T_{n}(\alpha _{0})>c_{n,\beta }(\delta
-\eta _{n}),\beta \in I_{n}\right) +\eta _{n} \\
&\leq &P_{n,\alpha _{0},\beta }\left( T_{n}(\alpha _{0})>c_{n,\beta }(\delta
-\eta _{n})\right) +\eta _{n} \\
&=&\delta -\eta _{n}+\eta _{n}=\delta .
\end{eqnarray*}%
Hence, the random critical value $c_{n,\eta _{n},Loh}(\delta )$ results in a
test that is guaranteed to be level $\delta $. In fact, its size can also be
lower bounded by $\delta -\eta _{n}$ provided there exists a $\beta
_{n}^{\max }(\delta -\eta _{n})$ satisfying $c_{n,\beta _{n}^{\max }(\delta
-\eta _{n})}(\delta -\eta _{n})=\sup_{\beta \in B}c_{n,\beta }(\delta -\eta
_{n})$: This follows since%
\begin{eqnarray}
&&\sup_{\beta \in B}P_{n,\alpha _{0},\beta }\left( T_{n}(\alpha
_{0})>c_{n,\eta _{n},Loh}(\delta )\right)  \notag \\
&\geq &\sup_{\beta \in B}P_{n,\alpha _{0},\beta }\left( T_{n}(\alpha
_{0})>\sup_{\beta \in B}c_{n,\beta }(\delta -\eta _{n})\right)  \notag \\
&=&\sup_{\beta \in B}P_{n,\alpha _{0},\beta }\left( T_{n}(\alpha
_{0})>c_{n,\beta _{n}^{\max }(\delta -\eta _{n})}(\delta -\eta _{n})\right) 
\notag \\
&\geq &P_{n,\alpha _{0},\beta _{n}^{\max }(\delta -\eta _{n})}\left(
T_{n}(\alpha _{0})>c_{n,\beta _{n}^{\max }(\delta -\eta _{n})}(\delta -\eta
_{n})\right)  \notag \\
&=&\delta -\eta _{n}.  \label{lower}
\end{eqnarray}%
The critical value (\ref{Loh}) (or asymptotic variants thereof) has also
been used in econometrics, e.g., by DiTraglia (2011), McCloskey (2011,
2012), and Romano, Shaikh, and Wolf (2014).

The test based on the random critical value $c_{n,\eta _{n},Loh}(\delta )$
may have size strictly smaller than $\delta $. This suggests that this test
will not improve over the conservative test based on $c_{n,\sup }(\delta )$
for \emph{all} values of $\beta $: We can expect that the test based on (\ref%
{Loh}) will sacrifice some power when compared with the conservative test (%
\ref{cons}) when the true $\beta $ is close to $\beta _{n}^{\max }(\delta )$
or $\beta _{n}^{\max }(\delta -\eta _{n})$; however, we can often expect a
power gain for values of $\beta $ that are "far away" from $\beta _{n}^{\max
}(\delta )$ and $\beta _{n}^{\max }(\delta -\eta _{n})$, as we then
typically will have that $c_{n,\eta _{n},Loh}(\delta )$ is smaller than $%
c_{n,\sup }(\delta )$. Hence, each of the two tests will typically have a
power advantage over the other in certain parts of the parameter space $B$.

It is thus tempting to try to construct a test that has the power advantages
of both these tests by choosing as a critical value the smaller one of the
two critical values, i.e., by choosing%
\begin{equation}
\hat{c}_{n,\eta _{n},\min }(\delta )=\min \left( c_{n,\sup }(\delta
),c_{n,\eta _{n},Loh}(\delta )\right)   \label{McC}
\end{equation}%
as the critical value. While both critical values $c_{n,\sup }(\delta )$ and 
$c_{n,\eta _{n},Loh}(\delta )$ lead to level $\delta $ tests, this is,
however, unfortunately not the case in general for the test based on the
random critical value (\ref{McC}). To see why, note that by construction the
critical value (\ref{McC}) satisfies 
\begin{equation*}
\hat{c}_{n,\eta _{n},\min }(\delta )\leq c_{n,\sup }(\delta ),
\end{equation*}%
and hence can be expected to fall under the wrath of Proposition \ref{prop_1}
given above. Thus it can be expected to not deliver a test that has level $%
\delta $, but has a size that exceeds $\delta $. So while the test based on
the random critical value proposed in (\ref{McC}) will typically reject more
often than the tests based on (\ref{cons}) or on (\ref{Loh}), it does so by
violating the size constraint. Hence it suffers from the same problems as
the parametric bootstrap test (\ref{par_boot}). [We make the trivial
observation that the lower bound (\ref{lower}) also holds if $\hat{c}%
_{n,\eta _{n},\min }(\delta )$ instead of $c_{n,\eta _{n},Loh}(\delta )$ is
used, since $\hat{c}_{n,\eta _{n},\min }(\delta )\leq c_{n,\eta
_{n},Loh}(\delta )$ holds.] As a point of interest we note that the
construction (\ref{McC}) has actually been suggested in the literature, see
McCloskey's (2011).\footnote{%
This construction is no longer suggested in McCloskey (2012).} In fact,
McCloskey (2011) suggested a random critical value $\hat{c}_{n,McC}(\delta )$
which is the minimum of critical values of the form (\ref{McC}) with $\eta
_{n}$ running through a finite set of values; it is thus less than or equal
to the individual $\hat{c}_{n,\eta _{n},\min }$'s, which exacerbates the
size distortion problem even further.

While Proposition 1 shows that tests based on random critical values like $%
c_{n,\hat{\beta}_{n}}(\delta )$ or $\hat{c}_{n,\eta _{n},\min }(\delta )$
will typically not have level $\delta $, it leaves open the possibility that
the overshoot of the size over $\delta $ may converge to zero as sample size
goes to infinity, implying that the test would then be at least \emph{%
asymptotically} of level $\delta $. In sufficiently "regular" testing
problems this will indeed be the case. However, for many testing problems
where nuisance parameters are present such as, e.g., testing post-model
selection, it turns out that this is typically \emph{not }the case: In the
next section we illustrate this by providing a prototypical example where
the overshoot does \emph{not} converge to zero for the tests based on $c_{n,%
\hat{\beta}_{n}}(\delta )$ or $\hat{c}_{n,\eta _{n},\min }(\delta )$, and
hence these tests are not level $\delta $ \emph{even} \emph{asymptotically}.

\section{An Illustrative Example}

In the following we shall -- for the sake of exposition -- use a very simple
example to illustrate the issues involved. Consider the linear regression
model 
\begin{equation}
y_{t}\quad =\quad \alpha x_{t1}+\beta x_{t2}+\epsilon _{t}\qquad (1\leq
t\leq n)  \label{e.1}
\end{equation}%
under the "textbook"\ assumptions that the errors $\epsilon _{t}$ are i.i.d.~%
$N(0,\sigma ^{2})$, $\sigma ^{2}>0$, and the nonstochastic $n\times 2$
regressor matrix $X$ has full rank (implying $n>1$) and satisfies $X^{\prime
}X/n$ $\rightarrow $ $Q>0$ as $n\rightarrow \infty $. The variables $y_{t}$, 
$x_{ti}$, as well as the errors $\epsilon _{t}$ can be allowed to depend on
sample size $n$ (in fact may be defined on a sample space that itself
depends on $n$), but we do not show this in the notation. For simplicity, we
shall also assume that the error variance $\sigma ^{2}$ is known and equals $%
1$. It will be convenient to write the matrix $(X^{\prime }X/n)^{-1}$ as 
\begin{equation*}
(X^{\prime }X/n)^{-1}\quad =\quad \left( 
\begin{array}{cc}
\sigma _{\alpha ,n}^{2} & \sigma _{\alpha ,\beta ,n} \\ 
\sigma _{\alpha ,\beta ,n} & \sigma _{\beta ,n}^{2}%
\end{array}%
\right) .
\end{equation*}%
The elements of the limit of this matrix will be denoted by $\sigma _{\alpha
,\infty }^{2}$, etc. It will prove useful to define $\rho _{n}=\sigma
_{\alpha ,\beta ,n}/(\sigma _{\alpha ,n}\sigma _{\beta ,n})$, i.e., $\rho
_{n}$ is the correlation coefficient between the least-squares estimators
for $\alpha $ and $\beta $ in model (\ref{e.1}). Its limit will be denoted
by $\rho _{\infty }$. Note that $\left\vert \rho _{\infty }\right\vert <1$
holds since $Q>0$ has been assumed.

As in Leeb and P\"{o}tscher (2005) we shall consider two candidate models
from which we select on the basis of the data: The unrestricted model
denoted by \emph{U} which uses both regressors $x_{t1}$ and $x_{t2}$, and
the restricted model denoted by \emph{R }which uses only the regressor $%
x_{t1}$ (and thus corresponds to imposing the restriction $\beta =0$). The
least-squares estimators for $\alpha $ and $\beta $ in the unrestricted
model will be denoted by $\hat{\alpha}_{n}(\emph{U})$ and $\hat{\beta}_{n}(%
\emph{U})$, respectively. The least-squares estimator for $\alpha $ in the
restricted model will be denoted by $\hat{\alpha}_{n}(\emph{R})$, and we
shall set $\hat{\beta}_{n}(\emph{R})=0$. We shall decide between the
competing models \emph{U} and \emph{R }depending on whether $|\sqrt{n}\hat{%
\beta}(\emph{U}_{n})/\sigma _{\beta ,n}|>c$ or not, where $c>0$ is a
user-specified cut-off point independent of sample size (in line with the
fact that we consider conservative model selection). That is, we select the
model $\hat{M}_{n}$ according to%
\begin{equation*}
\hat{M}_{n}=\left\{ 
\begin{array}{cc}
\emph{U} & \text{if \ }|\sqrt{n}\hat{\beta}_{n}(\emph{U})/\sigma _{\beta
,n}|>c, \\ 
\emph{R} & \text{otherwise.}%
\end{array}%
\right.
\end{equation*}%
We now want to test the hypothesis $H_{0}:\alpha =\alpha _{0}$ versus $%
H_{1}: $ $\alpha >\alpha _{0}$ and we insist, for better or worse, on using
the test-statistic%
\begin{eqnarray*}
T_{n}(\alpha _{0})=\left[ n^{1/2}\left( \hat{\alpha}(\emph{R})-\alpha
_{0}\right) /\left( \sigma _{\alpha ,n}\left( 1-\rho _{n}^{2}\right)
^{1/2}\right) \right] \boldsymbol{1}(\hat{M}_{n}=\emph{R}) && \\
+\left[ n^{1/2}\left( \hat{\alpha}(\emph{U})-\alpha _{0}\right) /\sigma
_{\alpha ,n}\right] \boldsymbol{1}(\hat{M}_{n} &=&\emph{U}).
\end{eqnarray*}%
That is, depending on which of the two models has been selected, we insist
on using the corresponding textbook test statistic (for the known-variance
case). While this could perhaps be criticized as somewhat simple-minded, it
describes how such a test may be conducted in practice when model selection
precedes the inference step. It is well-known that if one uses this
test-statistic and naively compares it to the usual normal-based quantiles
acting as if the selected model were given a priori, this results in a test
with severe size-distortions, see, e.g., Kabaila and Leeb (2006) and
references therein. Hence, while sticking with $T_{n}(\alpha _{0})$ as the
test-statistic, we now look for appropriate critical values in the spirit of
the preceding section and discuss some of the proposals from the literature.
Note that the situation just described fits into the framework of the
preceding section with $\beta $ as the nuisance parameter and $B=\mathbb{R}$.

Calculations similar to the ones in Leeb and P\"{o}tscher (2005) show that
the finite-sample distribution of $T_{n}(\alpha _{0})$ under $H_{0}$ has a
density that is given by%
\begin{eqnarray*}
&&h_{n,\beta }(u)=\Delta \left( n^{1/2}\beta /\sigma _{\beta ,n},c\right)
\phi \left( u+\rho _{n}\left( 1-\rho _{n}^{2}\right) ^{-1/2}n^{1/2}\beta
/\sigma _{\beta ,n}\right)  \\
&&+\left( 1-\Delta \left( \left( 1-\rho _{n}^{2}\right) ^{-1/2}\left(
n^{1/2}\beta /\sigma _{\beta ,n}+\rho _{n}u\right) ,\left( 1-\rho
_{n}^{2}\right) ^{-1/2}c\right) \right) \phi \left( u\right) ,
\end{eqnarray*}%
where $\Delta (a,b)=\Phi (a+b)-\Phi (a-b)$ and where $\phi $ and $\Phi $
denote the density and cdf, respectively, of a standard normal variate. Let $%
H_{n,\beta }$ denote the cumulative distribution function (cdf)
corresponding to $h_{n,\beta }$.

Now, for given significance level $\delta $, $0<\delta <1$, let $c_{n,\beta
}(\delta )=H_{n,\beta }^{-1}(1-\delta )$ as in the preceding section. Note
that the inverse function exists, since $H_{n,\beta }$ is continuous and is
strictly increasing as its density $h_{n,\beta }$ is positive everywhere. As
in the preceding section let%
\begin{equation}
c_{n,\sup }(\delta )=\sup_{\beta \in \mathbb{R}}c_{n,\beta }(\delta )
\label{cons_1}
\end{equation}%
denote the conservative critical value (the supremum is actually a maximum
in the interesting case $\delta \leq 1/2$ in view of Lemmata \ref{quant} and %
\ref{quant_2} in the Appendix). Let $c_{n,\hat{\beta}_{n}(\emph{U)}}(\delta )
$ be the parametric bootstrap based random critical value. With $\eta $
satisfying $0<\eta <\delta $, we also consider the random critical value 
\begin{equation}
c_{n,\eta ,Loh}(\delta )=\sup_{\beta \in I_{n}}c_{n,\beta }(\delta -\eta )
\label{Loh_1}
\end{equation}%
where%
\begin{equation*}
I_{n}=\left[ \hat{\beta}_{n}(\emph{U})\pm n^{-1/2}\sigma _{\beta ,n}\Phi
^{-1}(1-(\eta /2))\right] 
\end{equation*}%
is an $1-\eta $ confidence interval for $\beta $. [Again the supremum is
actually a maximum.] We choose here\emph{\ }$\eta $\emph{\ independent }of%
\emph{\ }$n$\emph{\ }as in McCloskey (2011, 2012) and DiTraglia (2011) and
comment on sample size dependent\emph{\ }$\eta $\emph{\ }below. Furthermore
define%
\begin{equation}
\hat{c}_{n,\eta ,\min }(\delta )=\min \left( c_{n,\sup }(\delta ),c_{n,\eta
,Loh}(\delta )\right) .  \label{McC_1}
\end{equation}%
Recall from the discussion in Section 1 that these critical values have been
used in the literature in the contexts of testing post-model-selection,
post-moment-selection, or post-model-averaging. Among the critical values $%
c_{n,\sup }(\delta )$, $c_{n,\hat{\beta}_{n}(\emph{U)}}(\delta )$, $%
c_{n,\eta ,Loh}(\delta )$, and $\hat{c}_{n,\eta ,\min }(\delta )$, we
already know that $c_{n,\sup }(\delta )$ and $c_{n,\eta ,Loh}(\delta )$ lead
to tests that are valid level $\delta $ tests. We next confirm -- as
suggested by the discussion in the preceding section -- that the random
critical values $c_{n,\hat{\beta}_{n}(\emph{U)}}(\delta )$ and $\hat{c}%
_{n,\eta ,\min }(\delta )$ (at least for some choices of $\eta $) do \emph{%
not} lead to tests that have level $\delta $ (i.e., their size is strictly
larger than $\delta $). Moreover, we also show that the sizes of the tests
based on $c_{n,\hat{\beta}_{n}(\emph{U)}}(\delta )$ or $\hat{c}_{n,\eta
,\min }(\delta )$ do \emph{not }converge to $\delta $ as $n\rightarrow
\infty $, implying that the asymptotic sizes of these tests exceed $\delta $%
. These results a fortiori also apply to any random critical value that does
not exceed $c_{n,\hat{\beta}_{n}(\emph{U)}}(\delta )$ or $\hat{c}_{n,\eta
,\min }(\delta )$ (such as, e.g., McCloskey's (2011) $\hat{c}_{n,McC}(\delta
)$ or $c_{n,\eta ,Loh^{\ast }}(\delta )$). In the subsequent theorem we
consider for simplicity only the case $\rho _{n}\equiv \rho $, but the
result extends to the more general case where $\rho _{n}$ may depend on $n$.

\begin{theorem}
\label{main}Suppose $\rho _{n}\equiv \rho \neq 0$ and let $0<\delta \leq 1/2$
be arbitrary. Then 
\begin{equation}
\inf_{n>1}\sup_{\beta \in \mathbb{R}}P_{n,\alpha _{0},\beta }\left(
T_{n}(\alpha _{0})>c_{n,\hat{\beta}_{n}(\emph{U})}(\delta )\right) >\delta .
\label{1}
\end{equation}%
Furthermore, for each fixed $\eta $, $0<\eta <\delta $, that is sufficiently
small we have 
\begin{equation}
\inf_{n>1}\sup_{\beta \in \mathbb{R}}P_{n,\alpha _{0},\beta }\left(
T_{n}(\alpha _{0})>\hat{c}_{n,\eta ,\min }(\delta )\right) >\delta .
\label{2}
\end{equation}
\end{theorem}

\begin{proof}
We first prove (\ref{2}). Introduce the abbreviation $\gamma =n^{1/2}\beta
/\sigma _{\beta ,n}$ and define $\hat{\gamma}(\emph{U})=n^{1/2}\hat{\beta}(%
\emph{U})/\sigma _{\beta ,n}$. Observe that the density $h_{n,\beta }$ (and
hence the cdf $H_{n,\beta }$) depends on the nuisance parameter $\beta $
only via $\gamma $, and otherwise is independent of sample size $n$ (since $%
\rho _{n}=\rho $ is assumed). Let $\bar{h}_{\gamma }$ be the density of $%
T_{n}(\alpha _{0})$ when expressed in the reparameterization $\gamma $.\emph{%
\ }As a consequence, the quantiles satisfy $c_{n,\beta }(v)=\bar{c}_{\gamma
}(v)$ for every $0<v<1$, where $\bar{c}_{\gamma }(v)=\bar{H}_{\gamma
}^{-1}(1-v)$ and $\bar{H}_{\gamma }$ denotes the cdf corresponding to $\bar{h%
}_{\gamma }$. Furthermore, for $0<\eta <\delta $, observe that $c_{n,\eta
,Loh}(\delta )=\sup_{\beta \in I_{n}}c_{n,\beta }(\delta -\eta )$ can be
rewritten as%
\begin{equation*}
c_{n,\eta ,Loh}(\delta )=\sup_{\gamma \in \left[ \hat{\gamma}(\emph{U})\pm
\Phi ^{-1}(1-(\eta /2))\right] }\bar{c}_{\gamma }(\delta -\eta ).
\end{equation*}%
Now define $\gamma ^{\max }=\gamma ^{\max }(\delta )$ as a value of $\gamma $
such that $\bar{c}_{\gamma ^{\max }}(\delta )=\bar{c}_{\sup }(\delta
):=\sup_{\gamma \in \mathbb{R}}\bar{c}_{\gamma }(\delta )$. That such a
maximizer exists follows from Lemmata \ref{quant} and \ref{quant_2} in the
Appendix. Note that $\gamma ^{\max }$ does not depend on $n$. Of course, $%
\gamma ^{\max }$ is related to $\beta _{n}^{\max }=\beta _{n}^{\max }(\delta
)$ via $\gamma ^{\max }=n^{1/2}\beta _{n}^{\max }/\sigma _{\beta ,n}$. Since 
$\bar{c}_{\sup }(\delta )=\bar{c}_{\gamma ^{\max }}(\delta )$ is strictly
larger than 
\begin{equation*}
\lim_{\left\vert \gamma \right\vert \rightarrow \infty }\bar{c}_{\gamma
}(\delta )=\Phi ^{-1}(1-\delta )
\end{equation*}%
in view of Lemmata \ref{quant} and \ref{quant_2} in the Appendix, we have
for all sufficiently small $\eta $, $0<\eta <\delta $, that 
\begin{equation}
\lim_{\left\vert \gamma \right\vert \rightarrow \infty }\bar{c}_{\gamma
}(\delta -\eta )=\Phi ^{-1}(1-(\delta -\eta ))<\bar{c}_{\sup }(\delta )=\bar{%
c}_{\gamma ^{\max }}(\delta ).  \label{cruc}
\end{equation}%
Fix such an $\eta $. Let now $\varepsilon >0$ satisfy $\varepsilon <\bar{c}%
_{\sup }(\delta )-\Phi ^{-1}(1-(\delta -\eta ))$. Because of the limit
relation in the preceding display, we see that there exists $M=M(\varepsilon
)>0$ such that for $\left\vert \gamma \right\vert >M$ we have $\bar{c}%
_{\gamma }(\delta -\eta )<\bar{c}_{\sup }(\delta )-\varepsilon $. Define the
set%
\begin{equation*}
A=\left\{ x\in \mathbb{R}:\left\vert x\right\vert >\Phi ^{-1}(1-(\eta
/2))+M\right\} .
\end{equation*}%
Then on the event $\left\{ \hat{\gamma}(\emph{U})\in A\right\} $ we have
that $\hat{c}_{n,\eta ,\min }(\delta )\leq \bar{c}_{\sup }(\delta
)-\varepsilon $. Furthermore, noting that $P_{n,\alpha _{0},\beta _{n}^{\max
}}\left( T_{n}(\alpha _{0})>c_{n,\sup }(\delta )\right) =P_{n,\alpha
_{0},\beta _{n}^{\max }}\left( T_{n}(\alpha _{0})>\bar{c}_{\sup }(\delta
)\right) =\delta $, we have%
\begin{eqnarray*}
&&\sup_{\beta \in \mathbb{R}}P_{n,\alpha _{0},\beta }\left( T_{n}(\alpha
_{0})>\hat{c}_{n,\eta ,\min }(\delta )\right) \geq P_{n,\alpha _{0},\beta
_{n}^{\max }}\left( T_{n}(\alpha _{0})>\hat{c}_{n,\eta ,\min }(\delta
)\right)  \\
&=&P_{n,\alpha _{0},\beta _{n}^{\max }}\left( T_{n}(\alpha _{0})>\bar{c}%
_{\sup }(\delta )\right) +P_{n,\alpha _{0},\beta _{n}^{\max }}\left( \hat{c}%
_{n,\eta ,\min }(\delta )<T_{n}(\alpha _{0})\leq \bar{c}_{\sup }(\delta
)\right)  \\
&\geq &\delta +P_{n,\alpha _{0},\beta _{n}^{\max }}\left( \hat{c}_{n,\eta
,\min }(\delta )<T_{n}(\alpha _{0})\leq \bar{c}_{\sup }(\delta ),\hat{\gamma}%
(\emph{U})\in A\right)  \\
&\geq &\delta +P_{n,\alpha _{0},\beta _{n}^{\max }}\left( \bar{c}_{\sup
}(\delta )-\varepsilon <T_{n}(\alpha _{0})\leq \bar{c}_{\sup }(\delta ),\hat{%
\gamma}(\emph{U})\in A\right) .
\end{eqnarray*}%
We are hence done if we can show that the probability in the last line is
positive and independent of $n$. But this probability can be written as
follows\thinspace \footnote{%
The corresponding calculation in previous versions of this paper had
erroneously omitted the term $\rho \left( 1-\rho ^{2}\right) ^{-1/2}\gamma
^{\max }$ from the expression on the far right-hand side of the subsequent
display. This is corrected here by accounting for this term. Alternatively,
one could drop the probability involving $\left\vert \hat{\gamma}(\emph{U}%
)\right\vert \leq c$ altogether from the proof and work with the resulting
lower bound.
\par
{}}%
\begin{eqnarray*}
&&P_{n,\alpha _{0},\beta _{n}^{\max }}\left( \bar{c}_{\sup }(\delta
)-\varepsilon <T_{n}(\alpha _{0})\leq \bar{c}_{\sup }(\delta ),\hat{\gamma}(%
\emph{U})\in A\right)  \\
&=&P_{n,\alpha _{0},\beta _{n}^{\max }}\left( \bar{c}_{\sup }(\delta
)-\varepsilon <T_{n}(\alpha _{0})\leq \bar{c}_{\sup }(\delta ),\hat{\gamma}(%
\emph{U})\in A,\left\vert \hat{\gamma}(\emph{U})\right\vert \leq c\right)  \\
&&+P_{n,\alpha _{0},\beta _{n}^{\max }}\left( \bar{c}_{\sup }(\delta
)-\varepsilon <T_{n}(\alpha _{0})\leq \bar{c}_{\sup }(\delta ),\hat{\gamma}(%
\emph{U})\in A,\left\vert \hat{\gamma}(\emph{U})\right\vert >c\right)  \\
&=&P_{n,\alpha _{0},\beta _{n}^{\max }}\left( \bar{c}_{\sup }(\delta )\geq
n^{1/2}\left( \hat{\alpha}(\emph{R})-\alpha _{0}\right) /\left( \sigma
_{\alpha ,n}\left( 1-\rho ^{2}\right) ^{1/2}\right) >\right.  \\
&&\qquad \qquad \qquad \left. \bar{c}_{\sup }(\delta )-\varepsilon ,\hat{%
\gamma}(\emph{U})\in A,\left\vert \hat{\gamma}(\emph{U})\right\vert \leq
c\right)  \\
&&+P_{n,\alpha _{0},\beta _{n}^{\max }}\left( \bar{c}_{\sup }(\delta )\geq
n^{1/2}\left( \hat{\alpha}(\emph{U})-\alpha _{0}\right) /\sigma _{\alpha
,n}>\right.  \\
&&\qquad \qquad \qquad \left. \bar{c}_{\sup }(\delta )-\varepsilon ,\hat{%
\gamma}(\emph{U})\in A,\left\vert \hat{\gamma}(\emph{U})\right\vert
>c\right)  \\
&=&\left[ \Phi (\bar{c}_{\sup }(\delta )+\rho \left( 1-\rho ^{2}\right)
^{-1/2}\gamma ^{\max })-\Phi (\bar{c}_{\sup }(\delta )+\rho \left( 1-\rho
^{2}\right) ^{-1/2}\gamma ^{\max }-\varepsilon )\right]  \\
&&\times \Pr \left( Z_{2}\in A,\left\vert Z_{2}\right\vert \leq c\right)
+\Pr \left( \bar{c}_{\sup }(\delta )\geq Z_{1}>\bar{c}_{\sup }(\delta
)-\varepsilon ,Z_{2}\in A,\left\vert Z_{2}\right\vert >c\right) ,
\end{eqnarray*}%
where we have made use of independence of $\hat{\alpha}(\emph{R})$ and $\hat{%
\gamma}(\emph{U})$, cf. Lemma A.1 in Leeb and P\"{o}tscher (2003), and of
the fact that $n^{1/2}\left( \hat{\alpha}(\emph{R})-\alpha _{0}\right) $ is
distributed as $N(-\sigma _{\alpha ,n}\rho \gamma ^{\max },\sigma _{\alpha
,n}^{2}\left( 1-\rho ^{2}\right) )$ under $P_{n,\alpha _{0},\beta _{n}^{\max
}}$. Furthermore, we have used the fact that $\left( n^{1/2}\left( \hat{%
\alpha}(\emph{U})-\alpha _{0}\right) /\sigma _{\alpha ,n},\hat{\gamma}(\emph{%
U})\right) ^{\prime }$ is under $P_{n,\alpha _{0},\beta _{n}^{\max }}$
distributed as $\left( Z_{1},Z_{2}\right) ^{\prime }$ where%
\begin{equation*}
\left( Z_{1},Z_{2}\right) ^{\prime }\sim N\left( (0,\gamma ^{\max })^{\prime
},\left( 
\begin{array}{cc}
1 & \rho  \\ 
\rho  & 1%
\end{array}%
\right) \right) ,
\end{equation*}%
which is a non-singular normal distribution since $\left\vert \rho
\right\vert <1$. It is now obvious from the final expression in the last but
one display that the probability in question is strictly positive and is
independent of $n$. This proves (\ref{2}).

We turn to the proof of (\ref{1}). Observe that $c_{n,\hat{\beta}_{n}(\emph{U%
})}(\delta )=\bar{c}_{\hat{\gamma}(\emph{U})}(\delta )$ and that%
\begin{equation*}
\bar{c}_{\sup }(\delta )=\bar{c}_{\gamma ^{\max }}(\delta )>\lim_{\left\vert
\gamma \right\vert \rightarrow \infty }\bar{c}_{\gamma }(\delta )=\Phi
^{-1}(1-\delta )
\end{equation*}%
in view of Lemmata \ref{quant} and \ref{quant_2} in the Appendix. Choose $%
\varepsilon >0$ to satisfy $\varepsilon <\bar{c}_{\sup }(\delta )-\Phi
^{-1}(1-\delta )$. Because of the limit relation in the preceding display,
we see that there exists $M=M(\varepsilon )>0$ such that for $\left\vert
\gamma \right\vert >M$ we have $\bar{c}_{\gamma }(\delta )<\bar{c}_{\sup
}(\delta )-\varepsilon $. Define the set%
\begin{equation*}
B=\left\{ x\in \mathbb{R}:\left\vert x\right\vert >M\right\} .
\end{equation*}%
Then on the event $\left\{ \hat{\gamma}(\emph{U})\in B\right\} $ we have
that $c_{n,\hat{\beta}_{n}(\emph{U})}(\delta )=\bar{c}_{\hat{\gamma}(\emph{U}%
)}(\delta )\leq \bar{c}_{\sup }(\delta )-\varepsilon $. The rest of the
proof is then completely analogous to the proof of (\ref{2}) with the set $A$
replaced by $B$.
\end{proof}

\begin{remark}
\normalfont(i) Inspection of the proof shows that (\ref{2}) holds for every $%
\eta $, $0<\eta <\delta $, that satisfies (\ref{cruc}).

(ii) It is not difficult to show that the suprema in (\ref{1}) and (\ref{2})
actually do \emph{not} depend on $n$.
\end{remark}

\begin{remark}
\label{rem}\normalfont If we allow $\eta $ to depend on $n$, we may choose $%
\eta =\eta _{n}\rightarrow 0$ as $n\rightarrow \infty $. Then the test based
on $\hat{c}_{n,\eta _{n},\min }(\delta )$ still has a size that strictly
overshoots $\delta $ for every $n$, but the overshoot will go to zero as $%
n\rightarrow \infty $. While this test then "approaches" the conservative
test that uses $c_{n,\sup }(\delta )$, it does not respect the level for any
finite sample size. [The same can be said for Loh's (1985) original proposal 
$c_{n,\eta _{n},Loh^{\ast }}(\delta )$, cf.~Footnote \ref{foot}.] Contrast
this with the test based on $c_{n,\eta _{n},Loh}(\delta )$ which holds the
level for each $n$, and also "approaches" the conservative test if $\eta
_{n}\rightarrow 0$. Hence, there seems to be little reason for preferring $%
\hat{c}_{n,\eta _{n},\min }(\delta )$ (or $c_{n,\eta _{n},Loh^{\ast
}}(\delta )$) to $c_{n,\eta _{n},Loh}(\delta )$ in this scenario where $\eta
_{n}\rightarrow 0$.
\end{remark}

\bigskip

\section{References}

\ \ \ Andrews, D.~W.~K. \& P. Guggenberger (2009): Hybrid and Size-Corrected
Subsampling Methods. \emph{Econometrica }77, 721-762.

Bickel, P.~J. \& K.~A. Doksum (1977): \textit{Mathematical Statistics: Basic
Ideas and Selected Topics. }Holden-Day, Oakland.

Berger, R.~L. \& D.~D. Boos (1994): P Values Maximized Over a Confidence Set
for the Nuisance Parameter. \emph{Journal of the American Statistical
Association }89, 1012-1016.

DiTraglia, F.~J. (2011):\ Using Invalid Instruments on Purpose:\ Focused
Moment Selection and Averaging for GMM. Working Paper, Version November 9,
2011.

Kabaila, P. \& H. Leeb (2006): On the Large-Sample Minimal Coverage
Probability of Confidence Intervals after Model Selection. \emph{Journal of
the American Statistical Association }101, 619-629.

Leeb, H. \& B.~M. P\"{o}tscher (2003): The Finite-Sample Distribution of
Post-Model-Selection Estimators and Uniform Versus Non-Uniform
Approximations. \emph{Econometric Theory }19, 100-142.

Leeb, H. \& B.~M. P\"{o}tscher (2005): Model Selection and Inference: Facts
and Fiction. \emph{Econometric Theory }21, 29-59.

Loh, W.-Y. (1985): A New Method for Testing Separate Families of Hypotheses. 
\emph{Journal of the American Statistical Association }80, 362-368.

Liu, C.-A. (2011): A Plug-In Averaging Estimator for Regressions with
Heteroskedastic Errors, Working Paper, Version October 29, 2011.

McCloskey, A. (2011): Powerful Procedures with Correct Size for Test
Statistics with Limit Distributions that are Discontinuous in Some
Parameters. Working Paper, Version October 2011.

McCloskey, A. (2012): Bonferroni-based Size Correction for Nonstandard
Testing Problems. Working Paper, Brown University.

Romano, J.~P. \& A. Shaikh, M. Wolf (2014): A Practical Two-Step Method for
Testing Moment Inequalities. Working Paper, University of Zurich.

Silvapulle, M.~J. (1996): A Test in the Presence of Nuisance Parameters. 
\emph{Journal of the American Statistical Association }91, 1690-1693.
(Correction, ibidem 92 (1997), 801.)

Williams, D.~A. (1970): Discrimination Between Regression Models to
Determine the Pattern of Enzyme Synthesis in Synchronous Cell Cultures. 
\emph{Biometrics }26, 23-32.

\bigskip

\appendix{}

\section{Appendix}

\begin{lemma}
\label{help}Suppose a random variable $\hat{c}_{n}$ satisfies $\Pr \left( 
\hat{c}_{n}\leq c^{\ast }\right) =1$ for some real number $c^{\ast }$ as
well as $\Pr \left( \hat{c}_{n}<c^{\ast }\right) >0$. Let $S$ be real-valued
random variable. If for every non-empty interval $J$ in the real line 
\begin{equation}
\Pr \left( S\in J\mid \hat{c}_{n}\right) >0  \label{pos_interval}
\end{equation}%
holds almost surely, then 
\begin{equation*}
\Pr \left( \hat{c}_{n}<S\leq c^{\ast }\right) >0.
\end{equation*}%
The same conclusion holds if in (\ref{pos_interval}) the conditioning
variable $\hat{c}_{n}$ is replaced by some variable $w_{n}$, say, provided
that $\hat{c}_{n}$ is a measurable function of $w_{n}$.
\end{lemma}

\begin{proof}
Clearly%
\begin{equation*}
\Pr \left( \hat{c}_{n}<S\leq c^{\ast }\right) =E\left[ \Pr \left( S\in (\hat{%
c}_{n},c^{\ast }]\mid \hat{c}_{n}\right) \right] =E\left[ \Pr \left( S\in (%
\hat{c}_{n},c^{\ast }]\mid \hat{c}_{n}\right) \boldsymbol{1}\left( \hat{c}%
_{n}<c^{\ast }\right) \right] ,
\end{equation*}%
the last equality being true since the first term in the product is zero on
the event $\hat{c}_{n}=c^{\ast }$. Now note that the first factor in the
expectation on the far right-hand side of the above equality is positive
almost surely by (\ref{pos_interval}) on the event $\left\{ \hat{c}%
_{n}<c^{\ast }\right\} $, and that the event $\left\{ \hat{c}_{n}<c^{\ast
}\right\} $ has positive probability by assumption.
\end{proof}

\bigskip

Recall that $\bar{c}_{\gamma }(v)$ has been defined in the proof of Theorem %
\ref{main}.

\begin{lemma}
\label{quant}Assume $\rho _{n}\equiv \rho \neq 0$. Suppose $0<v<1$. Then the
map $\gamma \rightarrow \bar{c}_{\gamma }(v)$ is continuous on $\mathbb{R}$.
Furthermore, $\lim_{\gamma \rightarrow \infty }\bar{c}_{\gamma
}(v)=\lim_{\gamma \rightarrow -\infty }\bar{c}_{\gamma }(v)=\Phi ^{-1}(1-v)$.
\end{lemma}

\begin{proof}
If $\gamma _{l}\rightarrow \gamma $ then $\bar{h}_{\gamma _{l}}$ converges
to $\bar{h}_{\gamma }$ pointwise on $\mathbb{R}$. By Scheff\'{e}'s Lemma, $%
\bar{H}_{\gamma _{l}}$ then converges to $\bar{H}_{\gamma }$ in total
variation distance. Since $\bar{H}_{\gamma }$ is strictly increasing on $%
\mathbb{R}$, convergence of the quantiles $\bar{c}_{\gamma _{l}}(v)$ to $%
\bar{c}_{\gamma }(v)$ follows. The second claim follows by the same argument
observing that $\bar{h}_{\gamma }$ converges pointwise to a standard normal
density for $\gamma \rightarrow \pm \infty $.
\end{proof}

\begin{lemma}
\label{quant_2}Assume $\rho _{n}\equiv \rho \neq 0$.

(i) Suppose $0<v\leq 1/2$. Then for some $\gamma \in \mathbb{R}$ we have
that $\bar{c}_{\gamma }(v)$ is larger than $\Phi ^{-1}(1-v)$.

(ii) Suppose $1/2\leq v<1$. Then for some $\gamma \in \mathbb{R}$ we have
that $\bar{c}_{\gamma }(v)$ is smaller than $\Phi ^{-1}(1-v)$.
\end{lemma}

\begin{proof}
Standard regression theory gives%
\begin{equation*}
\hat{\alpha}_{n}(\emph{U})=\hat{\alpha}_{n}(\emph{R})+\rho \sigma _{\alpha
,n}\hat{\beta}_{n}(\emph{U})/\sigma _{\beta ,n},
\end{equation*}%
with $\hat{\alpha}_{n}(\emph{R})$ and $\hat{\beta}_{n}(\emph{U})$ being
independent; for the latter cf., e.g., Leeb and P\"{o}tscher (2003), Lemma
A.1. Consequently, it is easy to see that the distribution of $T_{n}(\alpha
_{0})$ under $P_{n,\alpha _{0},\beta }$ is the same as the distribution of%
\begin{eqnarray*}
T^{\prime } &=&T^{\prime }(\rho ,\gamma )=\left( \sqrt{1-\rho ^{2}}W+\rho
Z\right) \boldsymbol{1}\left\{ \left\vert Z+\gamma \right\vert >c\right\} \\
&&+\left( W-\rho \frac{\gamma }{\sqrt{1-\rho ^{2}}}\right) \boldsymbol{1}%
\left\{ \left\vert Z+\gamma \right\vert \leq c\right\} ,
\end{eqnarray*}%
where, as before, $\gamma =n^{1/2}\beta /\sigma _{\beta ,n}$, and where $W$
and $Z$ are independent standard normal random variables.

We now prove (i): Let $q$ be shorthand for $\Phi ^{-1}(1-v)$ and note that $%
q\geq 0$ holds by the assumption on $v$. It suffices to show that $\Pr
\left( T^{\prime }\leq q\right) <\Phi (q)$ for some $\gamma $. We can now
write%
\begin{eqnarray*}
\Pr \left( T^{\prime }\leq q\right) &=&\Pr \left( \sqrt{1-\rho ^{2}}W+\rho
Z\leq q\right) -\Pr \left( \left\vert Z+\gamma \right\vert \leq c,W\leq 
\frac{q-\rho Z}{\sqrt{1-\rho ^{2}}}\right) \\
&&+\Pr \left( \left\vert Z+\gamma \right\vert \leq c,W\leq q+\frac{\rho
\gamma }{\sqrt{1-\rho ^{2}}}\right) \\
&=&\Phi (q)-\Pr (A)+\Pr (B).
\end{eqnarray*}%
Here, $A$ and $B$ are the events given in terms of $W$ and $Z$. Picturing
these two events as subsets of the plane (with the horizontal axis
corresponding to $Z$ and the vertical axis corresponding to $W$), we see
that $A$ corresponds to the vertical band where $|Z+\gamma |\leq c$,
truncated above the line where $W=(q-\rho Z)/\sqrt{1-\rho ^{2}}$; similarly, 
$B$ corresponds to the same vertical band $|Z+\gamma |\leq c$, truncated now
above the horizontal line where $W=q+\rho \gamma /\sqrt{1-\rho ^{2}}$.

We first consider the case where $\rho >0$ and distinguish two cases:

Case 1: $\rho c\leq \left( 1-\sqrt{1-\rho ^{2}}\right) q$.

In this case the set $B$ is contained in $A$ for every value of $\gamma $,
with $A\backslash B$ being a set of positive Lebesgue measure. Consequently, 
$\Pr (A)>\Pr (B)$ holds for every $\gamma $, proving the claim.

Case 2: $\rho c>\left( 1-\sqrt{1-\rho ^{2}}\right) q$.

In this case choose $\gamma $ so that $-\gamma -c\geq 0$, and, in addition,
such that also $(q-\rho (-\gamma -c))/\sqrt{1-\rho ^{2}}<0$, which is
clearly possible. Recalling that $\rho >0$, note that the point where the
line $W=(q-\rho Z)/\sqrt{1-\rho ^{2}}$ intersects the horizontal line $%
W=q+\rho \gamma /\sqrt{1-\rho ^{2}}$ has as its first coordinate $Z=-\gamma
+(q/\rho )(1-\sqrt{1-\rho ^{2}})$, implying that the intersection occurs in
the right half of the band where $|Z+\gamma |\leq c$. As a consequence, $\Pr
(B)-\Pr (A)$ can be written as follows:

\begin{equation*}
\Pr (B)-\Pr (A)=\Pr (B\backslash A)-\Pr (A\backslash B)
\end{equation*}%
where%
\begin{eqnarray*}
B\backslash A &=&\left\{ -\gamma +(q/\rho )(1-\sqrt{1-\rho ^{2}})\leq Z\leq
-\gamma +c,\right. \\
&&\left. (q-\rho Z)/\sqrt{1-\rho ^{2}}<W\leq q+\rho \gamma /\sqrt{1-\rho ^{2}%
}\right\}
\end{eqnarray*}%
and 
\begin{eqnarray*}
A\backslash B &=&\left\{ -\gamma -c\leq Z\leq -\gamma +(q/\rho )(1-\sqrt{%
1-\rho ^{2}}),\right. \\
&&\left. q+\rho \gamma /\sqrt{1-\rho ^{2}}<W\leq (q-\rho Z)/\sqrt{1-\rho ^{2}%
}\right\} .
\end{eqnarray*}

Picturing $A\backslash B$ and $B\backslash A$ as subsets of the plane as in
the preceding paragraph, we see that these events correspond to two
triangles, where the triangle corresponding to $A\backslash B$ is larger
than or equal (in Lebesgue measure) to that corresponding to $B\backslash A$%
. Since $\gamma $ was chosen to satisfy $-\gamma -c\geq 0$ and $(q-\rho
(-\gamma -c))/\sqrt{1-\rho ^{2}}<0$, we see that each point in the triangle
corresponding to $A\backslash B$ is closer to the origin than any point in
the triangle corresponding to $B\backslash A$. Because the joint Lebesgue
density of $(Z,W)$, i.e., the bivariate standard Gaussian density, is
spherically symmetric and radially monotone, it follows that $\Pr
(B\backslash A)-\Pr (A\backslash B)<0$, as required.

The case $\rho <0$ follows because $T^{\prime }(\rho ,\gamma )$ has the same
distribution as $T^{\prime }(-\rho ,-\gamma )$.

Part (ii)\ follows since $T^{\prime }(\rho ,\gamma )$ has the same
distribution as $-T^{\prime }(-\rho ,\gamma )$.
\end{proof}

\begin{remark}
\normalfont If $\rho _{n}\equiv \rho \neq 0$ and $v=1/2$, then $\bar{c}%
_{0}(1/2)=\Phi ^{-1}(1/2)=0$ since $\bar{h}_{0}$ is symmetric about zero.
\end{remark}

\begin{remark}
\normalfont If $\rho _{n}\equiv \rho =0$, then $T_{n}(\alpha _{0})$ is
standard normally distributed for every value of $\beta $, and hence $\bar{c}%
_{\gamma }(v)=\Phi ^{-1}(1-v)$ holds for every $\gamma $ and $v$.
\end{remark}

\end{document}